\documentclass[12pt]{amsart}
\usepackage{amssymb}
\usepackage{amsmath, amscd}
\usepackage{amsthm}
\usepackage{comment}
\usepackage[table,xcdraw]{xcolor}
\usepackage{ulem}
\usepackage{tikz}
\usepackage{float}
\usepackage{graphicx}
\usepackage{circuitikz}
\usetikzlibrary{positioning}
\usepackage[colorlinks=true, linkcolor=blue,urlcolor=blue]{hyperref}

\newtheorem{theorem}{Theorem}[section]

\newtheorem{proposition}[theorem]{Proposition}
\newtheorem{lemma}[theorem]{Lemma}

\newtheorem{corollary}[theorem]{Corollary}
\theoremstyle{definition}

\newtheorem{example}[theorem]{Example}

\newcommand{\bigzero}{\mbox{\normalfont\Large\bfseries 0}}

\begin{document}
	
\author[M. Doostalizadeh]{Mina Doostalizadeh}
\address{Department of Mathematics, Tarbiat Modares University, 14115-111 Tehran Jalal AleAhmad Nasr, Iran}
\email{d\_mina@modares.ac.ir;  m.doostalizadeh@gmail.com}	
 
\author[A. Moussavi]{Ahmad Moussavi}
\address{Department of Mathematics, Tarbiat Modares University, 14115-111 Tehran Jalal AleAhmad Nasr, Iran}
\email{moussavi.a@modares.ac.ir; moussavi.a@gmail.com}
\author[P. Danchev]{Peter  Danchev}
\address{Institute of Mathematics and Informatics, Bulgarian Academy of Sciences, 1113
Sofia, Bulgaria}
\email{ danchev@math.bas.bg; pvdanchev@yahoo.com} 
 
\title[Unit uniquely clean rings]{Unit uniquely clean rings}

\maketitle
  
\begin{abstract}
We define the class of {\it unit uniquely clean} rings ({\it UnitUC} for short), that is a  common generalization of  uniquely clean rings and strongly nil clean rings. Abelian {\it UnitUC} rings are uniquely clean and {\it UnitUC} rings with nil Jacobson radical are strongly nil clean. These rings also generalize the UUC and CUC rings, defined by Calugareanu-Zhou in Mediterranean J. Math. (2023), which are rings whose clean elements are uniquely clean. These rings are also represent a natural generalization of the Boolian rings in that a ring is {\it UnitUC} if, and only if, it is exchange and Boolean modulo the Jacobson radical.  The behavior of {\it UnitUC} rings under group ring and matrix ring extensions
is investigated. Several examples are provided to explain and delimit the results.
\end{abstract}

\section{Introduction}
Let $R$ be an associative ring with $1 \neq 0$. The notations of Jacobson radical, the group of units, the set of nilpotent elements and the set of all idempotent elements of $R$ are denoted by $J(R)$, $U(R)$, $N(R)$ and $Id(R)$, respectively. An element $r \in R$ is called
nil-clean if there is an idempotent $e \in R$ and a nilpotent $b \in R$ such that
$r = e+b$; the element $r$ is further called strongly nil-clean if the idempotent
and the nilpotent can be chosen such that $be = eb$. A ring is called nil-clean
(strongly nil-clean) if each of its elements is nil-clean (strongly nil-clean).
The notion of nil-clean rings was first introduced by Diesl in \cite{Diesl} and has been
studied by many researchers for the last decade. The study of nil-clean
rings was motivated by the study of clean rings. The influence on the structure of rings of properties defined elementwise is intensively studied in the literature. Clean rings and their generalizations, rings with special types of units, generalizations of commutative rings have been investigated in relation
to various global ring properties. Introduced by Nicholson
in \cite{nicolson} and \cite{nicolson2}, an element $a$ of $R$ is defined to be clean if $a = e + u$,
where $e \in R$ is an idempotent and $u \in U(R)$; moreover, $a$ is called strongly
clean if the idempotent and the unit can be chosen such that $eu = ue$. A
ring is defined to be (strongly) clean if each of its elements is (strongly)
clean. The class of clean rings contains semiperfect rings, unit-regular rings,
strongly $\pi$-regular rings, etc. (See \cite{Camillo}, \cite{Camillo2}, and \cite{Burgess}). Clean rings have
been extensively investigated in the past forty years with studies related to
topology, functional analysis, and the Kothe conjecture (see \cite{Ara},  \cite{Camillo3}, \cite{Han}, \cite{Lee}, \cite{Ster}, etc). Clean rings were initially studied in \cite{nicolson}
as a class of rings satisfying the exchange property which was introduced
by Crawley and Jonsson in 1964 in \cite{Crawley}.\\
 The class of clean rings is quite large and includes, for example, semiperfect rings. A ring is said to be uniquely clean if each element can be uniquely expressed as the sum of an idempotent and a unit. The concept of uniquely clean rings has been studied by Anderson and Camillo \cite{AC} in the case of commutative rings, and by Nicholson and Zhou \cite{NZ} in the case of noncommutative rings.

It was proven in \cite{NZ} that a ring $R$ is uniquely clean if, and only if, $R$ modulo its Jacobson radical $J(R)$ is a Boolean ring, idempotents can be lifted modulo $J(R)$, and the idempotents of $R$ are central.

On the other hand, by \cite{nicolson}, $a\in R$ is called {\it uniquely clean} if there exists a unique $e \in {\rm Id}(R)$ such that $a-e \in U(R)$. In particular, a ring $R$ is said to be {\it uniquely clean} (or just {\it UC} for short) if every element in $R$ is uniquely clean.
In a similar vein, expanding the first part of the above concepts, in \cite{zhou1} a ring is called {\it CUC} if any clean element is uniquely clean, and a ring is called {\it UUC} if any unit is uniquely clean.\par

 We say  the idempotents $e$ and $f$ are equivalent, denoted by $e \sim f$, in a ring $R$ if there exists $u \in U(R)$ such that $e = u^{-1}fu$. It is clear that if $e$ and $f$ are idempotents such that $e \sim 0$ and $f \sim 1$, then $e=0$ and $f=1$. Inspired by what we have stated so far, we call a ring $R$ a {\it unit uniquely clean ring} ({\it UnitUC ring} for short)   if it is clean, and if $e + u = f + v$ for two clean representations, then $e \sim f$.\par

The class of {\it unit uniquely clean} rings,  is a  common generalization of  uniquely clean rings and strongly nil clean rings. Abelian {\it UnitUC} rings are uniquely clean and {\it UnitUC} rings with nil Jacobson radical are strongly nil clean. These rings also generalize the UUC and CUC rings, defined by \\
Calugareanu-Zhou in Mediterranean J. Math. (2023), which are rings whose clean elements are uniquely clean. These rings are also represent a natural generalization of the Boolian rings in that a ring is {\it UnitUC} if, and only if, it is exchange and Boolean modulo the Jacobson radical. 
 
The relations between all of these notions are shown by the next diagrams:\\\\
 $$\indent\indent\indent\indent\indent\indent\indent\indent\indent\indent\indent\indent   CUC$$
 $$ \indent\indent\indent\indent\indent\indent\indent\indent\indent\indent\indent\indent   \downarrow$$
 $$\text{Uniquely clean}\longrightarrow \text{Uniquely strongly clean}\longrightarrow UnitUC\longrightarrow UUC$$ $$ \indent\indent\indent \uparrow$$
        $$ \indent\indent\indent \text{Strongly nil-clean}$$\\
 
  The behavior of {\it UnitUC} rings under group ring and matrix ring extensions
is investigated. Several examples are provided to explain and delimit the results.\par

 In the next section, we give some critical examples and properties of {\it UnitUC} rings and their extensions (see, for instance, Propositions~\ref{impor lemma} and \ref{triangular}). The subsequent third section is devoted to the exhibition of some other remarkable characterization behaviors of this class of rings under some standard extending procedures like matrix rings and group rings (see Theorems~\ref{series}, \ref{nil}, \ref{division}, \ref{19}, \ref{thm imp}). We also pose some challenging properties which, hopefully, will stimulate a further research on the subject.
\section{The Results}
 
 To begin, we present the following lemma, the proof of which is routine.

\begin{lemma}\label{prod}
    A direct product $\prod_{i \in I}R_i$ is   UnitUC if, and only if, each $R_i$ is   UnitUC.
\end{lemma}

\begin{lemma} \label{first lemma}
    Let $e$ and $f$ be idempotent elements such that $e - f \in U(R)$. Then, $(1 - e) \sim f$.
\end{lemma}

\begin{proof}
Let us assume that $u := e - f \in U(R)$. Then we have
   $$u^{-1}(1-e)u= u^{-1}(1-e)(e - f)=-u^{-1}(1-e)f=-u^{-1}(1-f -u)f=f.$$
\end{proof}

\begin{example}
    Idempotents and central nilpotents are   UnitUC in any ring.
\end{example}
\begin{proof}
If $e^2 = e$, then we can express $e$ as $(1 - e) + (2e - 1)$. Assume that $e = f + u$, where $f^2 = f$ and $u \in U(R)$. Therefore, $e - f \in U(R)$, and according to Lemma \ref{first lemma}, $(1 - e) \sim f$.

If $q \in Nil(R)$, then we can express $q$ as $1 + (q-1)$. Assume that $q = f + u$, where $f^2 = f$ and $u \in U(R)$. Therefore, $f=u(1-u^{-1}q) \in Id(R) \cap U(R)=\{ 1 \}$.
\end{proof}

\begin{corollary}
    Every Boolean ring is   UnitUC.
\end{corollary}

\begin{lemma}
    Let $R$ be   UnitUC. Then for every $r \in R$ and $e \in Id(R)$, we have $er(1-e) = f - e$ and $(1-e)re = f'-e$, where $f, f' \in \text{Id}(R)$ and $f \sim e \sim f'$.
\end{lemma}
\begin{proof}
Let $R$ be  UnitUC. Choose $a = 1 - (e + er(1 - e))$. One easily checks that $a = e + (1 - 2e - er(1-e))= (e+er(1-e))+ (1-2(e+er(1-e)))$. Clearly, $(1- 2e+er(1-e))^{-1}= (1-er(1-e))(1-2e)$, $(e+er(1-e))^2=e+er(1-e)$ and $(1-2(e+er(1-e)))^2=1$. By the uniqueness, we get $u^{-1}eu-e = er(1-e)$ for some $u \in U(R)$.  Similarly, there exists $v \in U(R)$ such that $v^{-1}ev - e = (1-e)re$.
\end{proof}

\begin{corollary}\label{ex iso}
   A ring  $R$ is abelian    UnitUC if, and only if, it is  uniquely clean ring.
\end{corollary}

\begin{proposition}\label{impor lemma}
    Let $R$ be a   UnitUC ring. The following hold:

(1) For every $0\neq e \in Id(R)$ and $u, v \in U(R)$, $u+v \neq e$. In particular, $R$ is a UUC ring.

(2) For any $u, v \in U(R)$, $u+v \neq 1$.

(3) For any $0\neq e,f \in Id(R)$, and $u,v \in U(eRe)$, $ u+v\neq f$.

(4) For any $0\neq e \in Id(R)$, and $u,v \in U(eRe)$, $ u+v\neq e$.

(5) For every ideal $I \subseteq J(R)$, $R/I$ is a    UnitUC ring. The converse holds, provided idempotents lift modulo $I$.

\end{proposition}

\begin{proof}
(1)  If $u+v=e$, we have $u=e-v$. Since $R$ is a unit uniquely clean ring, it must be that $e\sim 0$, so $e=0$.

(2) It is sufficient to assume $e=1$ in (1).

(3) Assuming $u+v=f$, then $u+(1-e)$, $v-(1-e) \in U(R)$, because $(u+(1-e)(u^{-1}+(1-e))= 1 = (v-(1-e))(v^{-1}-(1-e))$, so $(u+(1-e) + (v-(1-e)) = f$, which contradicts (1).

(4) Clearly, (3) implies.

(5) Since $R$ is clean, we have that $\overline{R}=R/I$ is also clean. If $\bar{a} \in \overline{R}$ such that there exist $\bar{e},\bar{f} \in Id(\overline{R})$ and $\bar{u},\bar{u'} \in U(\overline{R})$ with $\bar{a}=\bar{e}+\bar{u}=\bar{f}+\bar{u'}$. Since $I \subseteq J(R)$,  $u,u' \in U(R)$, and since $R$ is clean we can assume that $e,f \in Id(R)$. On the one hand, $a-(e+u), a-(f+u') \in J(R)$, so there exists $j,j' \in J(R)$ such that $a=e+(u+j)=f+(u'+j')$. Since $U(R)+J(R) \subseteq U(R)$, we have that $u+j , u'+j' \in U(R)$. Since $R$ is a   UnitUC ring, there exists $v \in U(R)$ such that $e=v^{-1}fv$, which results in $\bar{e}=\bar{v}^{-1}\bar{f}\bar{v}$.
\end{proof}

\begin{example} \label{matrix}
For any ring $S\neq 0$ and any integer $n \ge 2$, $M_n(S)$ is not a   UnitUC ring.
\end{example}
\begin{proof}
It can easily be verified that 

$$\begin{pmatrix}
1 & 0 & \cdots & 0 & 0\\
0 & 1 & \cdots & 0 & 0\\
\vdots & \vdots & \ddots & \vdots & \vdots\\
0 & 0 & \cdots & 1 & 0\\
0 & 0 & \cdots & 0 & 1\\
\end{pmatrix}=
\begin{pmatrix}
1 & 1 & \cdots & 0 & 0\\
0 & 1 & \cdots & 0 & 0\\
\vdots & \vdots & \ddots & \vdots & \vdots\\
0 & 0 & \cdots & 1 & 1\\
1 & 0 & \cdots & 0 & 0\\
\end{pmatrix} +
\begin{pmatrix}
0 & -1 & \cdots & 0 & 0\\
0 & 0 & \cdots & 0 & 0\\
\vdots & \vdots & \ddots & \vdots & \vdots\\
0 & 0 & \cdots & 0 & -1\\
-1 & 0 & \cdots & 0 & 1\\
\end{pmatrix},$$

which contradicts Proposition \ref{impor lemma}(2).
\end{proof}

\begin{lemma}\label{corner not}
    Let R be a   UnitUC ring. Then For any $n > 1$, there does not exist $0 \neq e \in Id(R)$ such that $eRe \cong M_n(S)$ for some ring $S$.
\end{lemma}
\begin{proof}
    The proof is straightforward using Example \ref{matrix} and Proposition \ref{impor lemma}(4).
\end{proof}

\begin{lemma}
    If $R$ is a   UnitUC ring, then $2 \in J(R)$.
\end{lemma}
\begin{proof}
    If $2 \notin J(R)$, there exists $0 \neq e=e^2 \in 2R$. Hence $e = 2b$, where $b \in R$. We assume that $b':=be$, then $eb' = b' = b'e$. Then
$u := 1 - 3e$ is a unit with inverse $(1 - e) - b'$. It follows from $1+(1-3e) = (1-e)+(1-2e)$ that $1-e \sim 1$, i.e., $e = 0$. This is a contradiction.
\end{proof}

\begin{lemma}\label{ess lemma for ex}
    Let $\phi: R \longrightarrow S$ be an epimorphism of rings such that $\text{Ker}(\phi) \subseteq J(R)$. Then:

1) If R is a   UnitUC ring, then S is also   UnitUC.

2) If S is   UnitUC and $R$ is clean, then $R$ is a   UnitUC ring.
\end{lemma}
\begin{proof}
(1) By Proposition \ref{impor lemma}(5), there is nothing left to prove.

(2) Suppose $ e+u = f+v $ where $ e, f \in \text{Id}(R)$ and $ u, v \in U(R) $. Then, there exists $ d \in U(S) $ such that $\phi(e) = d\phi(f)d^{-1} $. Since $ \phi $ is surjective, there exists $ k \in R $ such that $ \phi(k) = d $. Since $d \in U(S)$ and $\text{Ker}(\phi) \subseteq J(R)$, we have $ k \in U(R) $. So, $ \phi(e-kfk^{-1}) = 0 $. This implies that $ e-kfk^{-1} \in J(R) $. Since $2 \in J(S)$ and $ \text{Ker}(\phi) \subseteq J(R) $, we have $ 2 \in J(R)$, so $ e-(1-kfk^{-1}) \in U(R)$. Therefore, by Lemma \ref{first lemma}, we have $e \sim f $.
\end{proof}

\begin{proposition}\label{triangular}

 The following statements are equivalent:

(1) $R$ is a   UnitUC ring.

(2) $T_n(R)$ is   UnitUC for all $n \in \mathbb{N}$.

(3)  $T_n(R)$ is  UnitUC for some $n \in \mathbb{N}$.

\end{proposition}
\begin{proof}
(2) $\Rightarrow $ (3) is obvious.

(3) $\Rightarrow $ (1): Assume $I=\{ (a_{ij}) \in T_n(R) : a_{ii}=0, \text{ for each } 1\leq i \leq n\}$. Then $T_n(R)/I \cong R^n$, so by Proposition \ref{impor lemma}(5), we have that $R^n$ is a   UnitUC ring. Thus, Lemma \ref{prod} implies that $R$ is a   UnitUC  ring.

(1) $\Rightarrow $ (3): Assume $\phi : T_n(R) \longrightarrow R^n$ with $\phi((a_{i,j}))= (a_{11}, \ldots, a_{nn})$. Clearly, $\phi$ is a surjective homomorphism, and $\text{Ker}(\phi) \subseteq J(T_n(R))$. Moreover, since $R$ is clean, it follows that $T_n(R)$ is clean. Furthermore, by Lemma \ref{prod}, $R^n$ is a   UnitUC ring. Thus, Lemma \ref{ess lemma for ex} implies that $T_n(R)$ is a   UnitUC ring.
\end{proof}

Let $R$ be a ring and $\alpha : R \to R$ is a ring endomorphism and $R[[x, \alpha]]$ denotes the ring of skew formal power series over $R$; that is all formal power series in $x$ with coefficients from $R$ with multiplication defined by $xr = \alpha(r)x$ for all $r \in R$. In particular, $R[[x]] = R[[x, 1_R]]$ is the ring of formal power series over $R$.

\begin{proposition}\label{series}
let $R$ be a ring with an endomorphism $\alpha : R \to R$. Then $R[[x, \alpha]]$ is   UnitUC if, and only if, $R$ is   UnitUC  ring.
\end{proposition}
\begin{proof}
    Assume $I=R[[x, \alpha]]x $. Then $R[[x, \alpha]]/I \cong R$, so by Proposition \ref{impor lemma}(5), we have that $R$ is a   UnitUC ring. Conversely, we consider $ \phi: R[[x, \alpha]] \longrightarrow R $ defined by $ \sum a_ix^i \mapsto a_0 $. Then, by Lemma \ref{ess lemma for ex}, the proof is complete. Note that $ J(R[[x, \alpha]]) = J(R) + I $ and $ U(R[[x, \alpha]]) = U(R) + I $.
\end{proof}

\begin{example}
Let $R$ be a ring. Then the following are equivalent:

(1) $R$ is   UnitUC.

(2) $R[x, \alpha]/\left\langle x^n \right\rangle$ is   UnitUC for all $n \in \mathbb{N}$.

(3)  $R[x, \alpha]/\left\langle x^n \right\rangle$ is   UnitUC for Some $n \in \mathbb{N}$.
\end{example}

\begin{proof}
     We consider $ \phi: R[x, \alpha]/\left\langle x^n \right\rangle \longrightarrow R$ defined by $ \sum_{i=0}^{n-1} a_ix^i \mapsto a_0 $. Then, by Lemma \ref{ess lemma for ex}, the proof is complete. Note that $J(R[x, \alpha]/\left\langle x^n \right\rangle) = J(R) + R[x, \alpha]/\left\langle x^n \right\rangle x $ and $ U(R[x, \alpha]/\left\langle x^n \right\rangle) = U(R) + R[x, \alpha]/\left\langle x^n \right\rangle x $.
\end{proof}

\begin{theorem} \label{bool}
     If $R$ is a  UnitUC ring, then $R/J(R)$ is a Boolean ring. 
\end{theorem}

\begin{proof}  
We Show that $\overline{R}=R/J(R)$ is a reduced ring. Suppose $x^2=0$, where $0 \neq x \in \overline{R}$, then by \cite[ Theorem 2.1]{lev}, there exists $0\neq e \in Id(\overline{R})$ such that $e\overline{R}e \cong M_2(S)$, for a nontrivial ring $S$, which contradicts Lemma \ref{corner not}. Thus, $\overline{R}$ is reduced and therefore abelian. By Example \ref{ex iso}, $\overline{R}$ is uniquely clean. Hence, by \cite[Theorem 20]{NZ}, $\overline{R}$ is Boolean.
\end{proof}

\begin{corollary}\label{nil}
For a ring $R$, the following conditions are equivalent:

(1)  $R$ is   UnitUC and $J(R)$ is nil.

(2) $R$ is strongly uniquely clean and $J(R)$ is nil.

(3) $R$ is strongly $J$-clean and $J(R)$ is nil.

(4) $R$ is strongly nil-clean.
\end{corollary}
\begin{proof}
(4) $\Rightarrow$ (3), Based on \cite[Theorem 5.1.5]{chenbo}, we have $Nil(R) = J(R)$, so nothing remains to prove.

(3) $\Rightarrow$ (1), Suppose $R$ is a strongly $J$-clean ring. Then, for every $a \in R$, there exist $e \in Id(R)$ and $j \in J(R)$ such that $a = e + j$ and $ej = je$. Thus, we can write $a = (1 - e) + (2e - 1 + j)$, which is a clean representation of $a$. Now, suppose $a = f + v$ is another clean representation of $a$. Then, we have $e - f = v + j \in U(R)$. Therefore, by Lemma \ref{first lemma}, we have $(1 - e) \sim f$.

(1) $\Rightarrow$ (4), Given Theorem \ref{bool}, for every $a \in R$, $a-a^2 \in J(R)$. Since $J(R)$ is nil, we have $a-a^2 \in \text{Nil}(R)$. Therefore, by \cite[Theorem 5.1.1]{chenbo}, $R$ is strongly nil-clean.

(4) $\Rightarrow$ (2), Given \cite[Theorem 5.1.5]{chenbo}, we have $Nil(R) = J(R)$. Also, from \cite[Corollary 5.1.2]{chenbo}, for every $a \in R$, there exists a unique idempotent $e \in R$ such that $a-e \in \text{Nil}(R) = J(R)$ with $ea = ae$. Therefore, based on \cite[Theorem 17]{USC}, $R$ is strongly uniquely clean.

(2) $\Rightarrow$ (4), Based on \cite[Corollary 18]{USC}, $R/J(R)$ is Boolean. Therefore, similar to (1) $\Rightarrow$ (4), nothing remains to prove.
\end{proof}

\begin{corollary}
Suppose $R$ is an Artinian (or finite) ring. Then, the following conditions are equivalent:

(1) $R$ is strongly nil-clean. 

(2) $R$ is strongly uniquely clean.

(3) $R$ is   UnitUC clean. 

(4) $R$ is strongly $J$-clean.

\end{corollary}
\begin{theorem}\label{division}
(1) A division ring $R$ is   UnitUC  if, and only if, $R \cong \mathbb{F}_2$.

(2) A semisimple ring $R$ is  UnitUC if, and only if, $R \cong \mathbb{F}_2 \times \cdots \times \mathbb{F}_2$.

(3) A local ring $R$ is   UnitUC if, and only if, $R/J(R) \cong \mathbb{F}_2$.

(4) If a semilocal ring $R$ be   UnitUC then $R/J(R) \cong \mathbb{F}_2 \times \cdots \times \mathbb{F}_2$.
\end{theorem}

\begin{proof}
(1) The proof is straightforward in light of Theorem \ref{bool}.
    
(2) The proof is straightforward based on Wedderburn–Artin theorem, Lemma \ref{prod}, and Example \ref{matrix}.

(3) Assuming $R$ is local and unit uniquely clean, by Lemma \ref{impor lemma}, $R/J(R)$ is a unit uniquely clean division ring. Thus, (1) implies that $R/J(R) \cong \mathbb{F}_2$. Conversely, the result follows from \cite[Theorem 15]{NZ}.

(4) The proof is similar to (3).
\end{proof}
In \cite{cnz}, the question of when a group ring is uniquely clean is addressed and the following results are obtained: If the group ring $RG$ is uniquely clean, then $R$ is uniquely clean and $G$ is a 2-group; the converse holds if $G$ is locally finite. In the following proposition, we provide some sufficient conditions for a group ring to be a unit uniquely clean ring. The ring homomorphism $\epsilon: RG \to R$, $\sum r_gg \to \sum r_g$, is called the augmentation map, and the kernel $ker(\epsilon)$ is called the augmentation ideal of $RG$ and is denoted by $\Delta(RG)$. Note that $\Delta(RG)$ is the ideal of $RG$ generated by $\{1 - g : g \in G \}$.

\begin{theorem}\label{19} Let $R$ be a ring and $G$ be a group.\par
     (1) If the group ring $RG$ is   UnitUC, then $R$ is   UnitUC and $G$ is a 2-group.\par (2) If $R$ is a   UnitUC ring and $G$ is a locally finite 2-group, then $RG$ is a   UnitUC ring.
\end{theorem}
\begin{proof}
 (1)   Let $RG$ be   UnitUC. Since $R \cong RG/\Delta(RG)$, $R$ is a clean ring, which implies that $R$ is   UnitUC, as $R$ is a subring of $RG$. Now, since $RG/J(RG)$ is Boolean by \ref{bool}, implies that $1 - g \in J(RG)$ for all $g \in G$, hence $\Delta(RG) \subseteq J(RG)$. Thus, $G$ is a $p$-group and $p \in J(R)$ for some prime $p$ by \cite[Proposition 15(i)]{con}. But since $2 \in J(R)$, it follows that $p=2$.\par
 (2) By \cite[Theorem 4]{zhou2}, $RG$ is clean. Assume $e+u=f+v$, where $e,f \in Id(RG)$ and $u,v \in U(RG)$. Then $\epsilon(e)+\epsilon(u) = \epsilon(f)+\epsilon(v)$, where $\epsilon(e),\epsilon(f) \in Id(R)$ and $\epsilon(u),\epsilon(v) \in U(R)$. Thus, there exists $d \in U(R)$ such that $\epsilon(e)=d^{-1}\epsilon(f)d$. Let $h:=d1_G\in U(RG)$. Then $\epsilon(e-h^{-1}fh)=0$, so $e-h^{-1}fh \in \Delta(RG)$. Since $2 \in J(R)$, by \cite[Lemma 2]{zhou2}, we have $e-h^{-1}fh \in J(RG)$. Also, by \cite[Proposition 9]{con}, we have $2 \in J(R) \subseteq J(RG)$. Hence, $e-(1-h^{-1}fh) \in U(RG)$, so by Lemma \ref{first lemma}, we conclude that $e \sim h^{-1}fh$.
\end{proof}

\begin{theorem}\label{thm imp}
    A ring $R$ is a   UnitUC ring if, and only if,
    
(1) $R$ is exchange.

(2) $R/J(R)$ is Boolean.
\end{theorem}
\begin{proof}
($\Rightarrow$) Since $R$ is a clean ring, it is sufficient to show that $R/J(R)$ is Boolean. Let's assume that $R$ is a   UnitUC ring. By Proposition \ref{impor lemma}(3), we have $R/J(R)$ is also   UnitUC. Therefore, by Lemma \ref{impor lemma}(1), we have $R/J(R)$ is a UUC ring. Using \cite[Theorem 3.1]{zhou1}, we can conclude that $R/J(R)$ is Boolean.

($\Leftarrow$) For every $a \in R$, we have $a - a^2 \in J(R)$ because $R$ is an exchange ring. There exists $e \in Id(R)$ such that $j := a - e \in J(R)$. Therefore, we can write $a = (1 - e) + (2e - 1 + j)$. Clearly, $(2e - 1 + j) \in U(R)$. Now, suppose $a = f + u$, where $f \in Id(R)$ and $u \in U(R)$. Then, we have $e - f = u - j \in U(R)$. Therefore, by Lemma \ref{first lemma}, we have $(1 - e) \sim f$.
\end{proof}

\begin{lemma} \label{corner ring}
    If $R$ is a  UnitUC ring and $e^2 = e \in R$, then $eRe$ is   UnitUC.
\end{lemma}
\begin{proof}
By Theorem \ref{thm imp}, it suffices to show that $eRe$ is an exchange ring and $eRe/J(eRe)$ is Boolean. Since $R$ is an exchange ring, by \cite[Corollary 2.6]{nicolson}, $eRe$ is also exchange. Moreover, since $R/J(R)$ is Boolean, for every $a \in eRe$, we have $a - a^2 \in J(R) \cap 
 eRe$. However, by \cite[Theorem 21.10]{lam2001first}, we have $J(eRe) = J(R) \cap eRe$. Hence, $eRe/J(eRe)$ is Boolean.
\end{proof}

A set $\{e_{ij} : 1 \le i, j \le n\}$ of nonzero elements of $R$ is said to be a system of $n^2$ matrix units if $e_{ij}e_{st} = \delta_{js}e_{it}$, where $\delta_{jj} = 1$ and $\delta_{js} = 0$ for $j \neq s$. In this case, $e := \sum_{i=1}^{n} e_{ii}$ is an idempotent of $R$ and $eRe \cong M_n(S)$ where $S = \{r \in eRe : re_{ij} = e_{ij}r\textit{, for all } i, j = 1, 2, . . . , n\}$.

\begin{proposition} \label{dedkind finite}
    Every  UnitUC ring is Dedekind finite.
\end{proposition}
\begin{proof}
    If $R$ is not a Dedekind finite ring, then there exist elements $a, b \in R$ such that $ab = 1$ but $ba \neq 1$. Assuming $e_{ij} = a^i(1-ba)b^j$ and $e =\sum_{i=1}^{n}e_{ii}$, there exists a nonzero ring $S$ such that $eRe \cong M_n(S)$. However, according to Lemma \ref{corner ring}, $eRe$ is a unit uniquely clean ring, so $M_n(S)$ must also be a unit uniquely clean ring, which contradicts Example \ref{matrix}.
\end{proof}

Let $A, B$ be two rings and $M, N$ be the $(A, B)$-bimodule and $(B, A)$-bimodule, respectively. Also, we consider the bilinear maps $\phi : M\otimes_B N \to A$ and $\psi : N\otimes_AM \to B$ that apply to the following properties
$$Id_M \otimes_B \psi = \phi \otimes_A Id_M, \quad Id_N \otimes_A \phi = \psi \otimes_B Id_N .$$
For $m \in M$ and $n \in N$, we define $mn := \phi(m \otimes n)$ and $nm := \psi(n \otimes m)$.
Thus, the 4-tuple 
$R= \begin{pmatrix}
A & M \\
N & B
\end{pmatrix}$
becomes to an associative ring equipped with the obvious matrix operations, which is called a Morita context ring. Denote the two-sided
ideals $Im\phi$ and $Im\psi$ to $MN$ and $NM$, respectively, that are called the trace ideals of the Morita context.

\begin{theorem} \label{morita}
Let  $R= \begin{pmatrix}
A & M \\
N & B
\end{pmatrix}$
be a Morita context ring. Then $R$ is a   UnitUC ring if, and only if, $A$, $B$ are   UnitUC rings and $MN \subseteq J(A)$, $NM \subseteq J(B)$.
\end{theorem}

\begin{proof}
Put $e= \begin{pmatrix}
    1_A & 0 \\ 0 & 0
\end{pmatrix},$ 
Note that $e$ and $1 - e$ are idempotents, and there are canonical ring isomorphisms $A \cong eRe$ and $B \cong (1 - e)R(1 - e)$. So, by Lemma \ref{corner ring}, $A$ and $B$ are   UnitUC rings. With \cite[Theorem 2.5]{tang}, we have $J(R) = \begin{pmatrix}
   J(A) & M_0\\ N_0 & J(B) 
\end{pmatrix}$, where $M_0 = \{x \in M : xN \subseteq J(A)\}$ and $N_0 =\{ y \in N : yM \subseteq J(B)\}$, and also from theorem \ref{thm imp}, $ R/J(R)$ is a reduced, so $\begin{pmatrix}
   0 & M\\ 0 & 0 
\end{pmatrix}$, $\begin{pmatrix}
   0 & 0 \\ N & 0 
\end{pmatrix} \in J(R)$. Therefore, $\begin{pmatrix}
   0 & M\\ N & 0 
\end{pmatrix} \in J(R)$, so for every $m \in M$ and $n \in N$, we have $\begin{pmatrix} 0 & m\\ n & 0  \end{pmatrix}\begin{pmatrix} 0 & m\\ n & 0  \end{pmatrix} \in J(R)$, which implies $MN \subseteq J(A)$ and $NM \subseteq J(B)$.

If $A$ and $B$ are exchange, then by \cite[Theorem 2.7(1)]{tang}, $R$ is also exchange. Furthermore, with \cite[Lemma 3.1(1)]{tang}, $R/J(R) \cong A/J(R) \times B/J(B)$, since $A/J(R)$ and $B/J(B)$ are Boolean, it follows that $R/J(R)$ is also Boolean. Therefore, by Theorem \ref{thm imp}, we have that $R$ is a   UnitUC ring.

\end{proof}

\begin{corollary}
Let $R$ be a ring with an idempotent $e \in R$. The following conditions are equivalent:

(1) $R$ is a   UnitUC ring.

(2) $eRe$ and $(1 - e)R(1 - e)$ are   UnitUC rings, and $eR(1 - e),(1 - e)Re \subseteq J(R)$.

\end{corollary}
\begin{proof}
    By Lemma \ref{corner ring}, $eRe$ and $(1-e)R(1-e)$ are   UnitUC rings. Also, since $R/J(R)$ is Boolean, we have $\text{Nil}(R) \subseteq J(R)$. Therefore, $eR(1-e), (1-e)Re \subseteq \text{Nil}(R) \subseteq J(R)$.\\ Conversely, we know that $R \cong \begin{pmatrix}
        eRe & eR(1-e) \\ (1-e)Re & (1-e)R(1-e)
    \end{pmatrix}$. Hence, by Theorem \ref{morita}, nothing remains to prove.
\end{proof}

Let $R$, $S$ be two rings, and let $M$ be an $(R, S)$-bimodule, such that the operation $(rm)s = r(ms$) is valid for all $r \in R$, $m \in M$ and $s \in S$. Given such a bimodule $M$, we can set
$$
T(R, S, M) =
\begin{pmatrix}
 R& M \\
 0& S
\end{pmatrix} 
=
\left\{  
\begin{pmatrix}
 r& m \\
 0& s
\end{pmatrix} 
: r \in R, m \in M, s \in S
\right\}
$$
where it forms a ring with the usual matrix operations. The so-stated formal matrix $T(R, S, M$) is called a formal triangular matrix ring. In Theorem \ref{morita}, if we set $N =\{0\}$, then we will obtain the following corollary.

\begin{corollary} \label{formal tringular matrix}
Let $R$, $S$ be two rings and let $M$ be an $(R, S)$-bimodule. Then, the formal triangular matrix ring $T(R, S, M)$ is   UnitUC if, and only if, $R$ and $S$ are   UnitUC rings.
\end{corollary}

\begin{corollary}
Let $K$ be a ring and $n \ge 1$ is a natural number. Then, $T_n(K)$ is   UnitUC if, and only if,  $K$ is   UnitUC.
\end{corollary}

\begin{proof}
It is sufficient to take in Corollary \ref{formal tringular matrix}, $R=K$, $S=T_{n-1}(K)$ and $M=K^{n-1}$.
\end{proof}

Given a ring $R$ and a central element $s$ of $R$, the 4-tuple 
$\begin{pmatrix}
 R& R \\
 R& R
\end{pmatrix}$
becomes a ring with addition defined componentwise and with multiplication defined by
$$
\begin{pmatrix}
 a_1& x_1 \\
 y_1& b_1
\end{pmatrix}
\begin{pmatrix}
  a_2& x_2 \\
 y_2& b_2
\end{pmatrix}=
\begin{pmatrix}
 a_1a_2 + sx_1y_2& a_1x_2 + x_1b_2 \\
 y_1a_2 + b_1y_2& sy_1x_2 + b_1b_2
\end{pmatrix}.
$$
This ring is denoted by $K_s(R)$. A Morita context
$
\begin{pmatrix}
  A& M \\
 N& B
\end{pmatrix}
$
with $A = B = M = N = R$ is called a generalized matrix ring over $R$. It was observed in \cite{krylov2008isomorphism} that a ring $S$ is a generalized matrix ring over $R$ if, and only if, $S = K_s(R)$ for some $s \in C(R)$. Here $MN = NM = sR$, so that $MN \subseteq J(A) \Longleftrightarrow  s \in J(R)$, $NM \subseteq J(B) \Longleftrightarrow  s \in  J(R)$, and $MN, NM$ are nilpotent $\Longleftrightarrow  s$ is a nilpotent. Thus, Theorem \ref{morita} has the following consequence, too.

\begin{corollary} \label{formal Ks(R)}
Let $R$ be a ring, and $s \in C(R)$. Then the formal matrix ring $K_s(R)$ is   UnitUC if, and only if, $R$ is unit uniquely clean and $s \in J(R)$
\end{corollary}

\begin{corollary}
Let $R$ be a ring, then the formal matrix ring $K_2(R)$ is   UnitUC if, and only if, $R$ is   UnitUC.
\end{corollary}

Following \cite{zhou3}, for $n \ge 2$ and for $s \in C(R)$, the $n \times n$ formal matrix ring over $R$ defined by $s$, denoted as $M_n(R; s)$, is the set of all $n \times n$ matrices over $R$ with usual addition of matrices and with multiplication defined below: for $(a_{ij})$ and $(b_{ij})$ in $M_n(R;s)$,
$$(a_{ij} )(b_{ij} )=(c_{ij} ), \textit{ where } c_{ij} = \sum_{k=1}^{n}s^{\delta_{ikj}} a_{ik}b_{kj}.$$
Here $\delta_{ijk} = 1 + \delta_{ik} - \delta_{ij} - \delta_{jk}$ with $\delta_{ik}$, $\delta_{ij}$, $\delta_{jk}$ the Kronecker delta symbols. Note that a proper matrix ring is never   UnitUC (see Lemma \ref{matrix}).

\begin{theorem}
Let $R$ be a ring with $s \in C(R)$ and let $n \ge 2$. Then $M_n(R; s)$ is   UnitUC if, and only if, $R$ is   UnitUC and $s \in J(R)$.    
\end{theorem}
\begin{proof}
We prove the claim by induction on $n$. For $n=2$, we have $M_2(R; s)=k_{s^2}(R)$, now the claim is true by Corollary \ref{formal Ks(R)}. Assume $n>2$ and the claim holds for $M_{n-1}(R,s)$. Let $A=M_{n-1}(R,s)$. Then 
$M_{n-1}(R,s)=\begin{pmatrix}A & M \\ N & R\end{pmatrix}$
is a Morita context, where 
$$M= \begin{pmatrix} M_{1n} \\  M_{2n} \\ \vdots \\  M_{n-1,n}\end{pmatrix},\quad
N= \begin{pmatrix} M_{n1} & M_{n2}  \cdots  & M_{n,n-1}\end{pmatrix}, \quad
M_{n,i}=M_{i,n}=R,$$
 $\;1 \le i \le n-1.$

Moreover, for every $x \in M$, $y\in N$ we have
\begin{equation}
   xy= \begin{pmatrix}
s^2x_{1n}y_{n1} & sx_{1n}y_{n2} & \cdots & sx_{1n}y_{n,n-1} \\
sx_{2n}y_{n1} & s^2x_{2n}y_{n2} & \cdots & sx_{2n}y_{n,n-1} \\
\vdots & \vdots & \ddots & \vdots\\
sx_{n-1,n}y_{n1} & sx_{n-1,n}y_{n2} & \cdots & s^2x_{n-1,n}y_{n,n-1}
\end{pmatrix}\in A 
\end{equation}
\begin{equation}
    yx= \sum_{k=1}^{n-1}s^2y_{n,k}x_{k,n} \in R.
\end{equation}

If $M_n(R; s)$ is   UnitUC, then, by Theorem \ref{morita}, $R$ is   UnitUC and $NM \subseteq J(R)$. Thus, by (2), $s^2R = NM \subseteq J(R)$, so $s^2 \in J(R)$ As $R/J(R)$ is Boolean, $s \in J(R)$.

Conversely, suppose that $R$ is   UnitUC and $s \in J(R)$. Then $NM = s^2R \subseteq J(R)$. But, by (1) and \cite[Theorem 11]{zhou3}, $MN \subseteq sA \subseteq J(R)$. Hence $M_n(R; s)$ is   UnitUC by Theorem \ref{morita}.
\end{proof}

\section{ Extensions of unit uniquely clean rings }

We start the trivial extension of a   UnitUC ring. Let $R$ be a ring and $M$ a bi-module over $R$. The {\it trivial extension} of $R$ and $M$ is defined as
$$T(R, M) = \{(r, m) : r \in R \text{ and } m \in M\},$$
with addition defined componentwise and multiplication defined by
$$(r, m)(s, n) = (rs, rn + ms).$$
Notice that the trivial extension ${\rm T}(R, M)$ is isomorphic to the subring
$$\left\{ \begin{pmatrix} r & m \\ 0 & r \end{pmatrix} : r \in R \text{ and } m \in M \right\}$$
consisting of the formal $2 \times 2$ matrix rings $\begin{pmatrix} R & M \\ 0 & R \end{pmatrix}$ and, in particular, the isomorphism ${\rm T}(R, R) \cong R[x]/\left\langle x^2 \right\rangle$ is fulfilled. We also note that the Jacobson radical of the trivial extension ${\rm T}(R, M)$ is $$J({\rm T}(R, M)) = {\rm T}(J(R), M).$$

A Morita context is referred to as {\it trivial} if the context products are trivial, meaning that $MN = \{0\}$ and $NM = \{0\}$ (see, for instance, \cite[p. 1993]{mar}). In this case, we have the isomorphism
$$\begin{pmatrix} A & M \\ N & B \end{pmatrix} \cong {\rm T}(A \times B, M\oplus N),$$
where $\begin{pmatrix} A & M \\ N & B \end{pmatrix}$ represents a trivial Morita context, as stated in \cite{kos1}.

A good information gives also the following necessary and sufficient condition.

\begin{proposition}\label{pro3.41}
Let $R$ be a ring and $M$ a bi-module over $R$. Then, ${\rm T}(R, M)$ is   UnitUC if, and only if, $R$ is  UnitUC.
\end{proposition}

\begin{proof}
It is sufficient to define $\phi: T(R,M) \longrightarrow R$, with $\phi((r,m))=r$. Clearly, $\phi$ is an epimorphism with $\text{Ker}(\phi) \subseteq J(T(R,M))$. Therefore, with Lemma \ref{ess lemma for ex}, nothing else remains to be proven. Note that if $R$ is a clean ring, then $T(R,M)$ is also clean.
\end{proof}

The next criterion is also worthy of documentation.

\begin{corollary}\label{cor3.38}
Let $R=A~M \choose N~B$ be a trivial Morita context. Then, $R$ is   UnitUC if, and only if, both $A$ and $B$ are   UnitUC.
\end{corollary}

\begin{proof}
It is straightforward bearing in mind Propositions \ref{pro3.41} and Lemma \ref{prod}.
\end{proof}

Likewise, we can derive the following:

\begin{corollary}\label{cor3.43}
Let $R$ be a ring and $M$ a bi-module over $R$. Then, the following four statements are equivalent:
\begin{enumerate}
\item
$R$ is a   UnitUC ring.
\item
${\rm T}(R, M)$ is a   UnitUC ring.
\item
${\rm T}(R, R)$ is a   UnitUC ring.
\item
$\dfrac{R[x]}{\left\langle x^2 \right\rangle}$ is a   UnitUC ring.
\end{enumerate}
\end{corollary}

Assume that 
$$S_n(R) =\left\{(a_{ij}) \in T_n(R) : a_{11}= \cdots = a_{nn}\right\} ,$$
 $$S^{\prime} = \left \{\begin{pmatrix}
    0 & a\\ 0 & b
\end{pmatrix}
\in T_2(R) : a, b \in R \right \}$$, 
 and   $$S'' = \left \{
\begin{pmatrix}
   0 & 0 & a\\ 0 & 0 & b\\ 0 & 0 & c
\end{pmatrix}
\in T_3(R) : a, b, c \in R \right \}.$$ Then, clearly we have $S_2(R) = T(R,R), S_3(R) \cong T(T(R,R),S^{\prime})$, and  $S_4(R) \cong T(S_3(R),S'')$.\\ Therefore, by Propositions \ref{pro3.41}, we have that $R$ is  UnitUC if and only if, $R_3$ is  UnitUC if, and only if, $R_4$ is   UnitUC.

\begin{proposition}
    Let $R$ be a ring. Then $S_n(R)$ is   UnitUC if, and only if, $R$ is   UnitUC ring.
\end{proposition}
\begin{proof}
Method 1) Assume that $\phi: S_n(R)\longrightarrow  R$ with $\phi((a_{ij}))=a_{11}$. Then, $\phi$ is a surjective homomorphism with $\text{Ker}(\phi) \subseteq J(S_n(R))$. Therefore, by Lemma \ref{ess lemma for ex}, the proof is complete.

Method 2) Let $L_n(R) = \left\{
\begin{pmatrix}
    0 & \cdots & 0 & a_1 \\ 0 & \cdots & 0 & a_2 \\ \vdots & \ddots & \vdots & \vdots \\ 0 & \cdots & 0 & a_n  
\end{pmatrix}
\in T_n(R) : a_i \in R \right \}$. Then, with a simple calculation, we can show that $S_n(R) \cong T(S_{n-1}(R), L_{n-1}(R))$. Therefore, by induction on $n$ and Proposition \ref{pro3.41}, the proof is straightforward.
\end{proof}

Let $R$ be a ring. We introduced  the following rings as follows,  which will appear in: \par
P. Danchev, A. Javan, O. Hasanzadeh, A. Moussavi,  Rings with $u-1$ quasinilpotent for each unit $u$, $J.\, Alg.\, Appl.$ (2024).

\begin{align*} 
&A_{n,m}(R) =R[x,y | x^n=xy=y^m=0], \\
&B_{n,m}(R) =R\left\langle x,y | x^n=xy=y^m=0  \right\rangle, \\
&C_{n}(R) =R \langle x,y | x^2=\underbrace{xyxyx...}_{\text{$n-1$ words}}=y^2=0 \rangle.
\end{align*}

Wang in \cite{wang} introduced the matrix ring $S_{n,m}(R)$. Suppose $R$ is a ring, then the matrix ring $S_{n,m}(R)$ can be represented as

$$\left\{ \begin{pmatrix}
   a & b_1 & \cdots & b_{n-1} & c_{1n} & \cdots & c_{1 n+m-1}\\
   \vdots  & \ddots & \ddots & \vdots & \vdots & \ddots & \vdots \\
   0 & \cdots & a & b_1 & c_{n-1,n} & \cdots & c_{n-1,n+m-1} \\
   0 & \cdots & 0 & a & d_1 & \cdots & d_{m-1} \\
   \vdots  & \ddots & \ddots & \vdots & \vdots & \ddots & \vdots \\
   0 & \cdots & 0 & 0  & \cdots & a & d_1 \\
   0 & \cdots & 0 & 0  & \cdots & 0 & a
\end{pmatrix}  : a, b_i, d_j,c_{i,j} \in R \right\}.$$

Also, let $T_{n,m}(R)$ be 

$$\left\{ \left(\begin{array}{@{}c|c@{}}
  \begin{matrix}
  a & b_1 & b_2 & \cdots & b_{n-1} \\
  0 & a & b_1 & \cdots & b_{n-2} \\
  0 & 0 & a & \cdots & b_{n-3} \\
  \vdots & \vdots & \vdots & \ddots & \vdots \\
  0 & 0 & 0 & \cdots & a
  \end{matrix}
  & \bigzero \\
\hline
  \bigzero &
  \begin{matrix}
  a & c_1 & c_2 & \cdots & c_{m-1} \\
  0 & a & c_1 & \cdots & c_{m-2} \\
  0 & 0 & a & \cdots & c_{m-3} \\
  \vdots & \vdots & \vdots & \ddots & \vdots \\
  0 & 0 & 0 & \cdots & a
  \end{matrix}
\end{array}\right)\in T_{n+m}(R) : a, b_i,c_j \in R \right\} $$

and 
$$U_{n}(R)=\left\{ \begin{pmatrix}
   a & b_1 & b_2 & b_3 & b_4 & \cdots & b_{n-1} \\ 
   0 & a & c_1 & c_2 & c_3 & \cdots & c_{n-2} \\ 
   0 & 0 & a & b_1 & b_2 & \cdots & b_{n-3} \\
   0 & 0 & 0 & a & c_1 & \cdots & c_{n-4} \\ 
   \vdots & \vdots & \vdots & \vdots &  &  & \vdots \\
   0 &0 & 0 & 0 & 0 & \cdots & a
\end{pmatrix}\in T_{n}(R) :  a, b_i, c_j \in R \right\}.$$

We show in the following lemma the relations between these rings.
 
\begin{lemma} \label{section 3 lemma}
    Let $R$ be a ring and $m, n \in \mathbb{N}$. Then
    
(1) $A_{n,m}(R) \cong T_{n,m}(R)$.

(2) $B_{n,m}(R) \cong S_{n,m}(R)$.

(3) $C_{n}(R) \cong U_{n}(R)$.
\end{lemma}

\begin{proof}
(1) We assume $f = a+ \sum_{i=1}^{n-1}b_ix^i + \sum_{j=1}^{m-1}c_jy^j \in A_{n,m}(R)$. We define $\varphi:  A_{n,m}(R) \to  T_{n,m}(R)$ as 
$$\varphi(f)=
\left(\begin{array}{@{}c|c@{}}
  \begin{matrix}
  a & b_1 & b_2 & \cdots & b_{n-1} \\
  0 & a & b_1 & \cdots & b_{n-2} \\
  0 & 0 & a & \cdots & b_{n-3} \\
  \vdots & \vdots & \vdots & \ddots & \vdots \\
  0 & 0 & 0 & \cdots & a
  \end{matrix}
  &\bigzero \\
\hline
  \bigzero &
  \begin{matrix}
  a & c_1 & c_2 & \cdots & c_{m-1} \\
  0 & a & c_1 & \cdots & c_{m-2} \\
  0 & 0 & a & \cdots & c_{m-3} \\
  \vdots & \vdots & \vdots & \ddots & \vdots \\
  0 & 0 & 0 & \cdots & a
  \end{matrix}
\end{array}\right). $$
It can be shown that $\varphi$ is a ring isomorphism.

(2) We assume $f \in B_{n,m}(R)$ such that 
\begin{align*}
    f &=  a_{00}y^0x^0 + a_{01}y^0x^1 + \cdots + a_{0,n-1}y^0x^{n-1}  \\
      &+  a_{10}y^1x^0 + a_{11}y^1x^1 + \cdots + a_{1,n-1}y^1x^{n-1}   \\
      & \qquad \vdots   \qquad \qquad \vdots  \qquad \qquad \qquad \qquad  \vdots   \\
      &+ a_{m-1,0}y^{m-1}x^0 + a_{m-1,1}y^{m-1}x^1 + \cdots + a_{m-1,n-1}y^{m-1}x^{n-1}.    \\
\end{align*}
We define $\psi :  B_{n,m}(R) \to  S_{n,m}(R)$ as 
$$\psi(f)=
 \begin{pmatrix}
   a_{00} & a_{10} & \cdots & a_{m-1,0} & a_{m-1,1} & \cdots & a_{m-1,n-1}\\
   \vdots  & \ddots & \ddots & \vdots & \vdots & \ddots & \vdots \\
   0 & \cdots & a_{00} & a_{10} & a_{11} & \cdots & a_{1,n-1} \\
   0 & \cdots & 0 & a_{00} & a_{01} & \cdots & a_{0,n-1} \\
   \vdots  & \ddots & \ddots & \vdots & \vdots & \ddots & \vdots \\
   0 & \cdots & 0 & 0  & \cdots & a_{00} & a_{0,1} \\
   0 & \cdots & 0 & 0  & \cdots & 0 & a_{00}
\end{pmatrix}.$$

(3) We introduce the coefficients as follows,
$$f= \sum_{0\le i_j \le 1\atop 1 \le j \le n-1}d_{(i_1, \dots , i_{n-1})}\underbrace{y^{i_1}x^{i_2}y^{i_3}x^{i_4}...}_{\text{$n-1$ words}}\in C_{n}(R)$$

We define $\phi :  C_{n}(R) \to  S_{n,m}(R)$ as 
$$\phi(f)=
\begin{pmatrix}
   d_{(0,0,0, \dots,0)} & d_{(1,0,0, \dots,0)} & d_{(1,1,0, \dots,0)} & d_{(1,1,1, \dots,0)} & \cdots & d_{(1,1,1, \dots,1)}\\
   0 & d_{(0,0,0, \dots,0)} & d_{(0,1,0, \dots,0)} & d_{(0,1,1, \dots,0)} & \cdots & d_{(0,1,1, \dots,1)} \\
   0 & 0 & d_{(0,0,0, \dots,0)} & d_{(1,0,0, \dots,0)} & \cdots & d_{(1,\dots,1,0,0)}\\
   0 & 0 & 0 & d_{(0,0,0, \dots,0)} & \cdots & d_{(0,1,\dots,1,0,0)} \\
   \vdots  & \vdots & \vdots & \vdots & \ddots & \vdots \\
   0 &  0 & 0  & 0  & \cdots & d_{(0,0,0, \dots,0)}
\end{pmatrix}.$$
\end{proof}

\begin{example}
Let R be a ring, then we have: 

(1) $R\left[ x,y | x^2=xy=y^2=0  \right] \cong \left\{  
\begin{pmatrix}
    a_1 & a_2 & 0 & 0 \\
    0 & a_1 & 0 & 0 \\
    0 & 0 & a_1 & a_3 \\
    0 & 0 & 0 & a_1
\end{pmatrix} : a_i \in R \right\}.
$

(2) $R\left\langle x,y | x^2=xy=y^2=0  \right\rangle \cong \left\{  
\begin{pmatrix}
    a_1 & a_2 & a_3 \\
    0 & a_1 & a_4 \\
    0 & 0 & a_1
\end{pmatrix} : a_i \in R \right\}. $

(3) $R\left\langle x,y | x^2=xyx=y^2=0  \right\rangle \cong \left\{  
\begin{pmatrix}
    a_1 & a_2 & a_3 & a_4 \\
    0 & a_1 & a_5 & a_6 \\
    0 & 0 & a_1 & a_2 \\
    0 & 0 & 0 & a_1
\end{pmatrix} : a_i \in R \right\}\\ \cong T(T(R,R), M_2(R)).$
    
\end{example}

\begin{proposition}
Let R be a ring. Then, the following statements are equivalent:

(1) $R$ is a   UnitUC ring. 

(2) $A_{n,m}(R)$ is a   UnitUC ring.

(3) $B_{n,m}(R)$ is a   UnitUC ring.

(4) $C_{n}(R)$ is a   UnitUC ring.

\end{proposition}
\begin{proof}
    The proof is similar to that of Proposition \ref{pro3.41}, so remove the details.
\end{proof}


\bibliographystyle{amsplain}

\end{document}